\documentclass[pagesize,paper=A4,11pt,bibliography=totoc,oneside]{scrartcl}

\usepackage{lmodern}
\usepackage[T1]{fontenc}
\usepackage[tracking]{microtype}

\usepackage{graphicx}
\usepackage[dvipsnames]{xcolor}
\usepackage{amssymb,amsmath,amsthm}
\usepackage{bbm}
\usepackage{mathrsfs} % \mathscr{}
\usepackage{thmtools}
\usepackage{booktabs}
\usepackage[numbers,sort&compress]{natbib}
\usepackage[colorlinks=true,linkcolor=red,citecolor=blue]{hyperref}
\usepackage{orcidlink}

\newtheorem{thm}{Theorem}
\newtheorem{proposition}[thm]{Proposition}
\newtheorem{lemma}[thm]{Lemma}
\newtheorem{definition}[thm]{Definition}
\newtheorem{corollary}[thm]{Corollary}

\newtheorem{example}[thm]{Example}
\newtheorem{remark}[thm]{Remark}

\newcommand{\ZZ}{\mathbbm{Z}}
\newcommand{\QQ}{\mathbbm{Q}}
\newcommand{\RR}{\mathbbm{R}}

\newcommand{\defas}{\mathrel{\mathop:}=}
\newcommand{\abs}[1]{\left|#1\right|}
\newcommand{\set}[1]{\{#1\}}
\DeclareMathOperator{\rk}{rk}
\newcommand{\UM}[2]{\mathsf{U}_{#2,#1}} % uniform matroid
\newcommand{\MM}[2]{\mathsf{T}_{#2,#1}} % minimal matroid
\newcommand{\MP}[1]{\mathscr{P}_{#1}} % matroid polytope
\newcommand{\IP}[1]{\mathcal{I}_{#1}} % independent set polytope
\DeclareMathOperator{\conv}{conv}
\newcommand{\uv}{\mathsf{e}} % unit vector

\newcommand{\ehr}{\mathrm{ehr}}
\newcommand{\Tutte}{T}
\newcommand{\HN}{H} % harmonic number

\title{\vspace{-10mm}The Tutte polynomial determines the linear coefficient of the Ehrhart polynomial}
\newcommand{\email}[1]{\href{mailto:#1}{#1}}
\author{%
    \thanks{Mathematical Institute, University of Oxford, OX2 6GG, UK, \email{erik.panzer@maths.ox.ac.uk}}
    Erik Panzer
    \orcidlink{0000-0002-9897-5812}
}

\begin{document}

\maketitle

\begin{abstract}\vspace{-1cm}
    We prove a formula for the linear coefficient of the Ehrhart polynomial in terms of the Tutte polynomial. Furthermore, we characterize the matroids that extremize this Ehrhart coefficient, proving a conjecture of Ferroni.
\end{abstract}

\section{Introduction}
The Tutte polynomial $\Tutte_M\in\ZZ[x,y]$ of a matroid $M$ is the generating function of all subsets $A$ of $M$, counted by their rank $\rk(A)$ and corank $\ell(A)\defas\abs{A}-\rk(A)$:
\begin{equation*}
    \Tutte_M(x,y)=\sum_{A\subseteq M} (x-1)^{\rk(M)-\rk(A)}(y-1)^{\ell(A)}.
\end{equation*}
It encodes a lot of information and it has many applications and interpretations in graph and matroid theory \cite{BrylawskiOxley:TutteApp,Crapo:Tutte,EllisMonaghanMerino:ApplicationsI}. Furthermore, it can be computed rather efficiently \cite{BHKK:TutteVertex,HaagardPearceRoyle:CompTutte}.

The Ehrhart polynomial $\ehr_M\in \QQ[t]$ of a matroid counts the number of lattice points in dilates of the matroid polytope. For a subset $A$ of $M$ let $\uv_{A}\defas\sum_{i\in A}\uv_i\in\{0,1\}^n$ denote the vector with $1$s in the entries labelled by elements of $A$. The matroid polytope
\begin{equation*}
	\MP{M}=\conv\{\uv_B\colon \text{$B$ is a basis of $M$}\}\subset\RR^n
\end{equation*}
is the convex hull of the indicator vectors of all bases. It is a lattice polytope in $\RR^n$ with $n=\abs{M}$ the size of (the ground set of) $M$. Now $\ehr_M$ is the unique polynomial such that
\begin{equation*}
	\ehr_M(k)=\abs{\ZZ^n \cap k\MP{M}}
\end{equation*}
for every integer $k\geq 0$. An intriguing conjecture from \cite{DeLoeraHawsKoeppe:Ehrhart} states that all coefficients of $\ehr_M$ should be non-negative. While this was recently disproved \cite{Ferroni:MatNotEhr}, it stimulated intensive research on matroid Ehrhart polynomials and many interesting results have been established; see e.g.\ \cite{FerroniMoralesPanova:EhrPosLat,Ferroni:HyperPositive,FanLi:EhrhartSchubert} and references therein.
%For example, it is known that the linear coefficient $\ehr'_M(0)$ is always non-negative, see \cite[Theorem~4.5]{JochemkoRavichandran:PermuEhrPos} or \cite[Theorem~7.5]{CastilloLiu:ToddPermuto}.

Our main result is a combinatorial formula for the linear coefficient of the Ehrhart polynomial. Let $\HN_k=1+\frac{1}{2}+\ldots+\frac{1}{k}$ denote the harmonic numbers, setting $H_0\defas 0$.
\begin{thm}\label{thm:Dehr0}
    For every matroid on $n\geq1$ elements, we have the identity
    \begin{equation}\label{eq:ehr0-binom}
        \ehr_M'(0)
%        = \sum_{i,j} iR_{ij} (-1)^{r+i-j} \left.\frac{\td}{\td x}\right|_{x=0} \binom{x+\ell-1+i-j}{n-1}
%        = rH_{n-1}-\sum_{\substack{i,j\\-\ell< i-j<r\phantom{-}}} \frac{i}{(n-1)\binom{n-2}{\ell-1+i-j}} R_{ij}
        = \rk(M) \HN_{n-1}-\sum_{\emptyset \subsetneq A \subsetneq M} \frac{\rk(M)-\rk(A)}{(n-1)\binom{n-2}{\abs{A}-1}}
        \ .
    \end{equation}
\end{thm}
Note that the right-hand side is a linear combination of coefficients of the Tutte polynomial: If we define $R_{ij}(M)\in\ZZ$ so that $\Tutte_M(1+x,1+y)=\sum_{i,j} R_{ij} x^i y^j$, then
\begin{equation*}
    \sum_{\abs{A}=k} \big(\rk(M)-\rk(A)\big)=\sum_{i} i R_{i,k+i-\rk(M)}\,.
\end{equation*}\vspace{-5mm}
\begin{example}\label{ex:ehr0-UM}
    The uniform matroid $M=\UM{n}{r}$ of size $n$ and rank $r=n-\ell$ has $R_{i0}=\binom{n}{\ell+i}$, $R_{0j}=\binom{n}{r-j}$ and all other $R_{ij}=0$. The sum in Theorem~\ref{thm:Dehr0} is easily evaluated to
    \begin{equation*}
        \ehr'_{\UM{n}{r}}(0) 
        =n\HN_n-r\HN_r-\ell \HN_\ell
        .
        %= \sum_{i=1}^{r}i\binom{n}{\ell+i}(-1)^{r-i}\left.\frac{\td}{\td x}\right|_{x=0} \binom{x+\ell-1+i}{n-1}.
    \end{equation*}
    We confirmed that taking the derivative at zero of the formula for $\ehr_{\UM{n}{r}}(t)$---as given in \cite[Corollary~2.2]{Katzman:HilbertVeronese} or \cite[Theorem~2.1]{Ferroni:HyperPositive}---can be simplified to the same expression.
\end{example}

Two matroids can share the same Tutte polynomial and yet have different Ehrhart polynomials \cite[\S3]{DeLoeraHawsKoeppe:Ehrhart}. %\todo{There is something wrong with table 1 in loc. cit.---their linear coefficients differ for matroids with same Tutte!}
Theorem~\ref{thm:Dehr0} thus reveals a surprising simplicity. %, and it is specific to the \emph{linear} Ehrhart coefficient. In contrast, higher coefficients of $\ehr_M$ are not linear combinations of $R_{ij}$'s (see \S\ref{sec:calc}).
Our result complements the handful of previously known linear (or constant) relations between Ehrhart and Tutte polynomial coefficients:
\begin{itemize}
    \item Recall that $\Tutte_M(1,1)$ counts the bases of $M$ (hence the vertices of $\MP{M}$). From the definitions it thus follows that, for every matroid,
\begin{equation*}
    \ehr_M(0) = 1 \qquad\text{and}\qquad
    \ehr_M(1) = \Tutte_M(1,1).
\end{equation*}
    \item Ehrhart-Macdonald reciprocity \cite{Breuer:InvitationEhrhart} identifies $\ehr_M(-k)$ with the number of lattice points in the relative interior of $k\MP{M}$. For a connected matroid on $n\geq2$ elements, the dimension of $\MP{M}$ is positive, and so since all lattice points of $\MP{M}$ are vertices,
\begin{equation*}
    \ehr_M(-1) = 0.
\end{equation*}
\item Let $\beta(M)=\partial_x|_{x=0} \Tutte_M(x,0)\in \ZZ_{\geq 0}$ denote Crapo's invariant \cite{Crapo:HigherInvariant}. For every matroid with rank $r$ that has no loops or coloops, it was shown in \cite{CDBFLRVM:EhrhartBeta} that
\begin{equation*}
    \ehr'_M(-1) = \frac{\beta(M)}{(n-1)\binom{n-2}{r-1}} \,.
\end{equation*}
\end{itemize}
Computer calculations indicate that the five relations above (including Theorem~\ref{thm:Dehr0}) are the only linear relations between Ehrhart and Tutte coefficients that hold for all connected matroids (not counting Brylawski's linear identities \cite{BekeCsajiCsikvariPituk:ShortByrlawski} between coefficients of $\Tutte_M$ themselves). %; see \S\ref{sec:calc}.
In fact we first discovered the existence of an additional 5th relation by calculating the dimensions of the vector spaces of valuative functions spanned by the Tutte and Ehrhart coefficients. This experiment was motivated by the findings of \cite{CDBFLRVM:EhrhartBeta}.

We apply our theorem to characterize the extrema of the linear Ehrhart coefficient:
\begin{corollary}\label{cor:0-uniform}
    For every matroid $M$ of size $n$ and rank $r=n-\ell$, we have the inequalities
    \begin{equation*}
        0 %=\ehr'_{\UM{r}{r}\oplus\UM{\ell}{0}}(0)
        \leq\ehr'_M(0)
        \leq 
        %\ehr_{\UM{n}{r}}'(0) =
        n\HN_n-r\HN_r-\ell \HN_\ell.
    \end{equation*}
    Furthermore, the lower bound is achieved if and only if $M\cong \UM{r}{r}\oplus\UM{\ell}{0}$, and the upper bound is achieved if and only if $M\cong\UM{n}{r}$.
\end{corollary}
The lower bound $\ehr_M'(0)\geq 0$ was known before: \cite[Theorem~4.5]{JochemkoRavichandran:PermuEhrPos} and \cite[Theorem~7.5]{CastilloLiu:ToddPermuto} proved that the linear Ehrhart coefficient is non-negative for all generalized permutohedra. These results are more general, but for the case of matroid polytopes, our proof is simpler.

The upper bound of Corollary~\ref{cor:0-uniform} proves a special case of \cite[Conjecture~1.5]{Ferroni:EhrhartMinimal}. That conjecture claims that every coefficient of $\ehr_M$ is bounded from above by the corresponding coefficient for the uniform matroid $\UM{n}{r}$ with the same size and rank as $M$. Proofs have been found when $M$ is a sparse paving Schubert matroid \cite[Theorem~1.10]{FanLi:EhrhartSchubert}, a paving matroid \cite[Theorem~1.4]{DMV:PanhandlePavingChain}, or any lattice path matroid \cite[Theorem~4.1]{FerroniMoralesPanova:EhrPosLat}. Our result extends these---albeit only for the linear Ehrhart coefficient---to all matroids.

Furthermore, \cite[Conjecture~1.5]{Ferroni:EhrhartMinimal} proposed that for \emph{connected} matroids $M$, each coefficient of $\ehr_M$ is bounded from below by the corresponding coefficient for the minimal matroid $\MM{n}{r}$ with the same size and rank as $M$. These matroids $\MM{n}{r}=\UM{r+1}{r}\oplus_2 \UM{\ell+1}{1}$ are the unique matroids that achieve the minimal number $n\ell+1$ of bases among all connected matroids with size $n$ and rank $1\leq r<n$, see \cite[Theorem~5]{Murty:NumberBases} or \cite[Theorem~3.2]{Dinolt:ExNonSep}.\footnote{Murty denotes $\MM{r}{n}$ as $F_{r,n}$ and showed that they are graphic: they arise from the cycle graph on $r+1$ vertices with one edge replaced by $\ell$ parallel edges \cite[Figure~1]{Murty:NumberBases}. See also \cite[\S2]{Ferroni:EhrhartMinimal}.}

 This lower bound of \cite[Conjecture~1.5]{Ferroni:EhrhartMinimal} fails in general \cite{Ferroni:MatNotEhr}. However, if we consider only the \emph{linear} Ehrhart coefficient, we can prove it for all (connected) matroids:
\begin{proposition}\label{prop:MM}
    For all connected matroids $M$ of size $n\geq 2$ and rank $r$, we have
    \begin{equation*}
        \HN_{r-1} + \HN_{\ell-1} + \frac{1}{(n-1)\binom{n-2}{r-1}} \leq \ehr_M'(0).
    \end{equation*}
    Moreover, equality holds if and only if $M\cong \MM{n}{r}$.
\end{proposition}

We give a quick proof of Theorem~\ref{thm:Dehr0} in \S\ref{sec:ehr0-tutte} using linearity of the Ehrhart coefficient over Minkowski sums. Our proofs of Corollary~\ref{cor:0-uniform} and Proposition~\ref{prop:MM} in \S\ref{sec:bounds} are inductive and exploit the contraction-deletion recursion of the Tutte polynomial---similar to the strategy employed in \cite{Murty:NumberBases}.

Throughout this paper, binomial coefficients $\binom{n}{k}$ with $k>n$ are defined to be $0$.

\subsection*{Acknowledgments}
I thank Alex Dong and Lyn Risso for upcoming joint work on the calculation of matroid Ehrhart polynomials, and Lyn Risso also for computer experiments suggesting the minimizer in the loopless case of Corollary~\ref{cor:ehr0-coloopless}.
Erik Panzer is funded as a Royal Society University Research Fellow through grant {URF{\textbackslash}R{\textbackslash}251041}.
For the purpose of Open Access, the author has applied a CC BY public copyright licence to any Author Accepted Manuscript version arising from this submission.
No generative AI tools were used at any stage during the research and writing of this paper.

\newpage

\section{Proof of the main theorem}\label{sec:ehr0-tutte}
Following \cite[\S4]{JochemkoRavichandran:PermuEhrPos}, we exploit that the linear coefficient $\ehr'_M(0)$ of the Ehrhart polynomial is a linear function with respect to the Minkowski sum $P+Q=\set{p+q\colon p\in P, q\in Q}$ of lattice polytopes \cite[Corollary~23]{BoroczkyLudwig:MinkowskiLattice}: For any lattice polytopes $P,Q$ and integers $s,t\geq 0$,
\begin{equation*}
    \ehr_{sP+tQ}'(0) = s\cdot \ehr_P'(0) + t\cdot \ehr_Q'(0)
    .
\end{equation*}
As a generalized permutohedron, a matroid polytope has a unique representation as a signed Minkowski sum of simplices.
Let $\widetilde{\beta}(M)=(-1)^{\rk(M)+1}\beta(M)\in\ZZ$ denote the signed version of Crapo's invariant \cite{Crapo:HigherInvariant},
\begin{equation*}%\label{eq:beta}
  \beta(M)=(-1)^{\rk(M)} \sum_{A\subseteq M} (-1)^{\abs{A}} \rk(A).
\end{equation*}
Denote the standard simplex in coordinates $I\subseteq\{1,\ldots,n\}$ by $\Delta_I=\conv\{\uv_i\colon i\in I\}\subset \RR^n$. Then \cite[Theorem~2.5]{ArdilaBenedettiDoker:MatroidVolumes} showed that
\begin{equation*}
	\MP{M} = \sum_{A\subset M} \widetilde{\beta}(M/A) \Delta_{M\setminus A}.
\end{equation*}
It does not matter if $A=M$ is included in this sum, since $\widetilde{\beta}(\emptyset)=0$ for the empty matroid. Summands with size $\abs{A}=n-1$ only contribute when their complement $\set{i}=M\setminus A$ is a coloop of $M$ (otherwise, $M/A$ will be a loop and $\widetilde{\beta}(M/A)=0$). In this case, $\widetilde{\beta}(M/A)=1$ so the effect of the summand $A$ is to translate the polytope $\MP{M}$ by the vector $\Delta_{M\setminus A}=\set{\uv_i}$. However, a translation does not affect the Ehrhart polynomial.

The values of $\ehr'_{P}(0)$ for simplices $P=\Delta_{M\setminus A}$ are well-known \cite[Proposition~4.3]{JochemkoRavichandran:PermuEhrPos}. To summarize, \cite[Corollary~23]{BoroczkyLudwig:MinkowskiLattice} and \cite[Theorem~2.5]{ArdilaBenedettiDoker:MatroidVolumes} imply
\begin{corollary}\label{cor:ehr0-beta}
    Let $H_k=1+\ldots+\frac{1}{k}$ denote the harmonic numbers. For every matroid $M$,
    \begin{equation*}
        \ehr_M'(0)=\sum_{\substack{A\subset M\\ \abs{A}\leq n-2}} \HN_{n-\abs{A}-1} \widetilde{\beta}(M/A).
    \end{equation*}
\end{corollary}
Since $M/A$ has 2 or more elements, duality of the beta invariant applies so that $\beta(M/A)=\beta( (M/A)^\star)=\beta(M^\star|_{M\setminus A})$ for the dual matroid $M^\star$ of $M$. Since $\ehr_M=\ehr_{M^\star}$ is invariant under duality, we can thus rewrite Corollary~\ref{cor:ehr0-beta} as
\begin{equation*}
    \ehr_M'(0)=\sum_{\substack{A\subseteq M\\ \abs{A}\geq 2}} \HN_{\abs{A}-1} (-1)^{\ell(A)+1}\beta(M|_A).
\end{equation*}
Here $\ell(A)=\abs{A}-\rk(A)$ denotes the corank of $A$ in $M$ (which is equal to the rank of $M\setminus A$ in $M^\star$).
Inserting the definition of $\beta$, we get
\begin{equation*}
    \ehr_M'(0)=-\sum_{\substack{A\subseteq M\\ \abs{A}\geq 2}} \HN_{\abs{A}-1} \sum_{B\subseteq A} (-1)^{\abs{A\setminus B}}\rk(B).
\end{equation*}
The empty set $B=\emptyset$ does not contribute since $\rk(\emptyset)=0$. The entire matroid $B=M$ contributes $-r\HN_{n-1}$ (provided that $n\geq 2$) where $r\defas\rk(M)$. Exchanging the order of sums, for fixed $B\neq \emptyset,M$ of size $b\neq 0,n$ the sum over $A$ is
\begin{align*}
    \sum_{\substack{A\supseteq B \\ \abs{A}\geq 2}}  (-1)^{\abs{A\setminus B}} \HN_{\abs{A}-1}
    &= \sum_{i=0}^{n-b} \binom{n-b}{i} (-1)^i \left( \HN_{b-1} + \sum_{k=b}^{b+i-1} \frac{1}{k}\right) \\
    &= \HN_{b-1} (1-1)^{n-b} + \sum_{k=0}^{n-1-b}\frac{1}{b+k} \sum_{i=k+1}^{n-b} \binom{n-b}{i} (-1)^i \\
    &=\sum_{k=0}^{n-1-b}\frac{1}{b+k} \binom{n-1-b}{k} (-1)^{k+1} 
    = - \frac{1}{b\binom{n-1}{b}}\ .
\end{align*}
Here we set $i=\abs{A\setminus B}$ and used well-known binomial coefficient summation identities; the identity used in the last step is \cite[Corollary~2.2]{SuryWangZhao:ReciBin}. In total, this leads to the formula
\begin{equation*}
    \ehr_M'(0)=-r\HN_{n-1} + \sum_{b=1}^{n-1} \frac{1}{(n-1)\binom{n-2}{b-1}} \sum_{\substack{B\subset M\\\abs{B}=b}} \rk (B).
\end{equation*}
Although we assumed $n\geq2$ in the above derivation, this formula stays correct for the two matroids $\UM{1}{0}$, $\UM{1}{1}$ with $n=1$: In this case, $\ehr_M(t)=1$ constant ($\MP{M}$ is a point) and our formula gives $\ehr_M'(0)=0$ with our convention $\HN_0=0$.

To arrive at the form stated in Theorem~\ref{thm:Dehr0}, it remains only to note that
\begin{equation}\label{eq:simple-harmonic-sum}
    \sum_{b=1}^{n-1} \frac{1}{(n-1)\binom{n-2}{b-1}}  r\binom{n}{b}
    =r \sum_{b=1}^{n-1} \frac{n}{b(n-b)}
    =r \sum_{b=1}^{n-1} \left(\frac{1}{b} + \frac{1}{n-b}\right)
    =2r \HN_{n-1}.
\end{equation}
\begin{remark}
    Another natural polytope associated to a matroid is the independent set polytope $\IP{M}$; see \cite[\S4]{ArdilaBenedettiDoker:MatroidVolumes} for details. It is also a signed Minkowski sum of simplices, computed in \cite[Theorem~4.4]{ArdilaBenedettiDoker:MatroidVolumes}, and the above calculation is straightforwardly adapted to this case. For the linear Ehrhart coefficient of the independent set polytope we thus obtain
    \begin{equation}
        \ehr'_{\IP{M}}(0)
        =\sum_{A\subset M} \HN_{n-\abs{A}} \widetilde{\beta}(M/A)
        = r\HN_n - \sum_{\emptyset\neq A\subset M} \frac{r-\rk(A)}{n\binom{n-1}{\abs{A}-1}}.
    \end{equation}
    We conclude that the inequality from \cite[Corollary~4.12]{ArdilaBenedettiDoker:MatroidVolumes} is a strictly weaker constraint than \cite[Corollary~4.10]{ArdilaBenedettiDoker:MatroidVolumes}. More precisely, $\ehr'_{\IP{M}}(0)-\ehr'_M(0)\geq \frac{r}{n}$ is always non-negative:
    \begin{equation*}
        \ehr'_{\IP{M}}(0)-\ehr'_M(0)=\sum_{A\subsetneq M} \frac{\widetilde{\beta}(M/A)}{n-\abs{A}}
        = \frac{r}{n} + \sum_{\emptyset\neq A\subsetneq M} \frac{r-\rk(A)}{n\binom{n-1}{\abs{A}}}\,.
    \end{equation*}
\end{remark}

\section{Bounds on the linear Ehrhart coefficient}\label{sec:bounds}
To bound the linear Ehrhart coefficient, we consider the integers that appear in the numerator in the sum on the right-hand side of Theorem~\ref{thm:Dehr0}.
\begin{definition}\label{def:ak}
    For any matroid $M$ of size $n$ and rank $r$ and any integer $0\leq k\leq n$, set
\begin{equation}\label{eq:ak}
    a_k(M)\defas\sum_{\substack{A \subseteq M\\ \abs{A}=k}} \big(r-\rk(A)\big)=\sum_{i-j=r-k} i R_{ij}
    \quad\in\quad\ZZ_{\geq0} \ .
\end{equation}
\end{definition}
The values $a_0(M)=r$ (from $A=\emptyset$) and $a_n(M)=0$ (from $A=M$) do not depend on $M$ (when $r$ is fixed) and do not enter our formula for $\ehr'_M(0)$. Recall that the coefficients $R_{ij}(M)\in\ZZ_{\geq0}$ of $\Tutte_M(1+x,1+y)=\sum_{i,j} R_{ij} x^i y^j$ count the number
\begin{equation*}
    R_{ij}=\abs{\{A\subseteq E\colon \rk(A)=r-i\ \text{and}\ \abs{A}=r-i+j\}}
\end{equation*}
of subsets of rank $r-i$ and size $r-i+j$. See remark~\ref{rem:CameronFink} for another interpretation of $a_k$.

\begin{lemma}\label{lem:ak-lower}
    For any matroid of size $n$ and rank $r$, we have $a_k(M)\geq (r-k)\binom{n}{k}$ for all $k\leq r$. Furthermore, the equality $a_r(M)=0$ holds if and only if $M\cong \UM{n}{r}$.
\end{lemma}
\begin{proof}
    The sum $\sum_{i-j=r-k} R_{ij}=\binom{n}{k}$ counts all subsets of size $\abs{A}=k$. Among sequences $R_{ij}\geq 0$ with this constraint, expression \eqref{eq:ak} is minimized when the entire sum $\binom{n}{k}$ is concentrated in the entry $R_{ij}$ with the minimal possible value for $i=r-k+j \geq 0$. This configuration is
    \begin{equation} \label{eq:Urn-Rij}
        R_{ij} = 0\quad\text{for $i,j\geq 1$;}\qquad
        R_{i0} = \binom{n}{r-i},\qquad\text{and}\qquad
        R_{0j} = \binom{n}{r+j}.        
    \end{equation}
    It yields the values $a_k=(r-k)\binom{n}{k}$ for $k\leq r$ and $a_k=0$ for $k\geq r$, proving the lower bound. If $a_r(M)=\sum_{i} i R_{ii}=0$, we must have $R_{ii}=0$ for all $i>0$ and thus all $R_{00}=\binom{n}{r}$ subsets of size $r$ have rank $r$. Therefore, any subset $A\subseteq M$ of size $\abs{A}\leq r$ is independent. This completely determines the rank function of $M$ and it follows that $M\cong\UM{n}{r}$ is the uniform matroid. Note that $\UM{n}{r}$ precisely realises the configuration \eqref{eq:Urn-Rij}.
\end{proof}
\begin{lemma}\label{lem:ak-upper}
    For any matroid of size $n$ and rank $r$, we have $a_k(M)\leq r\binom{n-1}{k}$ for all $k$. Furthermore, an equality $a_k(M)=r\binom{n-1}{r}$ at any $0<k<n$ implies that $M\cong \UM{r}{r}\oplus\UM{\ell}{0}$.
\end{lemma}
\begin{proof}
    Pick any basis $B$ of $M$. Then every subset $A\subseteq M$ has rank $\rk(A)\geq \rk(A\cap B)=\abs{A\cap B}$. Summing over $s=\abs{A\cap B}$ and setting $\ell=n-r=\abs{M\setminus B}$, we get the estimate
    \begin{equation*}
        a_k(M)=\sum_{s+t=k} \sum_{\substack{A'\subseteq B \\ \abs{A'}=s}} \sum_{\substack{A''\subseteq M\setminus B \\ \abs{A''}=t}} \Big(r-\rk(A'\sqcup A'')\Big)
        \leq \sum_{s+t=k} \binom{r}{s} \binom{\ell}{t} (r-s)
        =r\binom{n-1}{k}
    \end{equation*}
    by $(r-s)\binom{r}{s}=r\binom{r-1}{s-1}$ and the Chu-Vandermonde identity.
    
    An equality $a_k(M)=r\binom{n-1}{k}$ requires that $\rk(A)=\abs{A\cap B}$ holds for all bases $B$ and for all $k$-subsets $A$ of $M$. Suppose that $M$ has another basis $B'\neq B$, then via the basis exchange axiom we can find one of the form $B'=(B\setminus \{i\})\sqcup \{j\}$ for some distinct elements $i\in B$ and $j\in M\setminus B$. When $0<k<n$, we can choose a subset $A\subset M$ of size $k$ with $i\in A$ and $j\notin A$, so that $\abs{A\cap B}=1+\abs{A\cap B'}$. This contradicts $\abs{A\cap B}=\rk(A)=\abs{A\cap B'}$, so $B$ must be the only basis of $M$ and therefore $M\cong\UM{r}{r}\oplus\UM{\ell}{0}$.
\end{proof}
Note that at any fixed size $n$ and rank $r=n-\ell$, the following properties are equivalent:
\begin{itemize}
    \item $M\cong \UM{r}{r}\oplus \UM{\ell}{0}$ (i.e.\ every element of $M$ is a loop or a coloop),
    \item $\MP{M}$ is a point (i.e.\ $M$ has exactly one basis).
\end{itemize}
In this case, $\ehr_M(t)=1$ is the constant function and thus $\ehr'_M(t)=0$. From the proof of lemma~\ref{lem:ak-upper} we see that this matroid satisfies all equalities $a_k(M)=r\binom{n-1}{k}$, and one checks easily that the sum in Theorem~\ref{thm:Dehr0} thus indeed reproduces $\ehr'_M(t)=0$ in this case.

We now have all ingredients to prove Corollary~\ref{cor:0-uniform}.
\begin{proof}[{Proof of Corollary~\ref{cor:0-uniform}}]
    When $r=0$ or $r=n$, then there is only one matroid ($M\cong\UM{n}{0}$ or $M\cong\UM{n}{n}$) and the claim is trivial. When $n>r>0$, Definition~\ref{def:ak} and Theorem~\ref{thm:Dehr0} give
    \begin{equation}\label{eq:ehr0-ak}
        \ehr_M'(0)=r \HN_{n-1} - \sum_{k=1}^{n-1} \frac{a_k(M)}{(n-1)\binom{n-2}{k-1}}\,.
    \end{equation}
    So $\ehr'_{\UM{r}{r}\oplus\UM{\ell}{0}}(0)\leq\ehr'_M(0)\leq\ehr'_{\UM{n}{r}}(0)$ follows via $a_k(\UM{n}{r})\leq a_k(M)\leq a_k(\UM{r}{r}\oplus\UM{\ell}{0})$ from lemmas~\ref{lem:ak-lower} and \ref{lem:ak-upper}. To saturate the upper (lower) bound for $\ehr_M'(0)$, we must saturate the lower (upper) bounds for $a_k(M)$ simultaneously for all $1\leq k\leq n-1$, since they all contribute with strictly negative coefficient to $\ehr_M'(0)$. This includes in particular $a_r(M)$, and therefore lemma~\ref{lem:ak-lower} (and also lemma~\ref{lem:ak-upper}) characterize such extremal $M$.
\end{proof}
\begin{remark}\label{rem:CameronFink}
    In \cite[Theorem~3.2]{CameronFink:TutteLattice} the Tutte polynomial is related to a polynomial $Q'_M(x,y)=\sum_{i,j} (-1)^{n-1-i-j} b_{ij}x^i y^j\in\ZZ[x,y]$ of degree $n-1$ whose coefficients $b_{ij}\geq 0$ count the cells in a mixed subdivision of a polytope
    \cite[Theorm~1.4]{CameronFink:TutteLattice}. Chasing through these relations, this gives a geometric interpretation of $a_k$, since we find
    \begin{equation*}
        a_k(M)= (r-k) b_{n-1-k,k}+b_{n-1-k,k-1}
        =(r-k) \binom{n-1}{k}+b_{n-1-k,k-1}.
    \end{equation*}
\end{remark}

In order to approach Proposition~\ref{prop:MM}, we will strengthen the inequalities from lemma~\ref{lem:ak-upper},
\begin{equation*}
    a_k(M)\leq a_k(\UM{r}{r}\oplus\UM{\ell}{0}) = r\binom{n-1}{k}.
\end{equation*}
These are optimal if we allow arbitrary matroids $M$ of size $n$ and rank $r$, but they can be improved in the case when $M$ is constrained. We will consider the classes of loopless matroids, coloopless matroids, and connected matroids.
\begin{lemma}\label{lem:contract-delete}
    If $e\in M$ is neither a loop nor a coloop of $M$, then for all $k>0$ we have that
    \begin{equation*}
        a_k(M)=a_k(M\setminus e)+a_{k-1}(M/e)\,.
    \end{equation*}
\end{lemma}
\begin{proof}
    The contraction-deletion recursion of the Tutte polynomial gives $R_{ij}(M)=R_{ij}(M\setminus e)+R_{ij}(M/e)$. Inserting this into \eqref{eq:ak} yields the claimed identity, because $\rk(M/e)=\rk(M)-1$ and $\rk(M\setminus e)=\rk(M)$.
\end{proof}
\begin{lemma}\label{lem:upper-noloops}
    Let $M$ be a loopless matroid of size $n$ and rank $r\geq1$. Then for all $k\geq0$,
    \begin{equation*}
        a_k(M)\leq a_k(\UM{r-1}{r-1}\oplus \UM{\ell+1}{1})
        %= (r-1)\binom{n-1}{k} + \binom{r-1}{k}
        = r\binom{n-1}{k} -\binom{n-1}{k} + \binom{r-1}{k}
        \,.
    \end{equation*}
    Furthermore, equality at any $2\leq k\leq n-1$ implies that $M\cong\UM{r-1}{r-1}\oplus\UM{\ell+1}{1}$.
\end{lemma}
\begin{proof}
    At $r=1$ the only loopless matroid is the bond, so $M\cong \UM{n}{1}\cong \UM{0}{0}\oplus \UM{\ell+1}{1}$. In this case, $\rk(A)=r=1$ for every non-empty subset $A\subseteq M$, hence $a_k(M)=0$ for all $k>0$ and $a_0(M)=1$. Thus $a_k(M)=\binom{0}{k}=(r-1)\binom{n-1}{k}+\binom{r-1}{k}$ for all $k$.
    
    At $r=n$, again there is only a single matroid $M\cong\UM{n}{n}\cong\UM{r-1}{r-1}\oplus\UM{1}{\ell+1}$ and the claim reduces to lemma~\ref{lem:ak-upper}, because then $(r-1)\binom{n-1}{k}+\binom{r-1}{k}=r\binom{n-1}{k}$.

    For $2\leq r\leq n-1$ there are several isomorphism classes of matroids and we proceed by induction on $r$. Since $\ell>0$ in this case, $M$ must have an element $e$ which is not a coloop. We can thus apply lemma~\ref{lem:contract-delete} to $e$. Since deletion cannot create loops, the induction hypothesis applies to $a_k(M\setminus e)$ and together with lemma~\ref{lem:ak-upper} we obtain
    \begin{align*}
        a_k(M) &\leq a_k(\UM{r-1}{r-1}\oplus\UM{\ell}{1}) + a_{k-1}(\UM{r-1}{r-1}\oplus\UM{\ell}{0})
        \\ &= (r-1)\binom{n-2}{k} + \binom{r-1}{k} + (r-1)\binom{n-2}{k-1}
        \\ &= (r-1)\binom{n-1}{k} + \binom{r-1}{k}.
    \end{align*}
    This proves the claimed bound. To get equality at some $2\leq k\leq n-1$ we must in particular have $a_{k-1}(M/e)=a_{k-1}(\UM{r-1}{r-1}\oplus\UM{\ell}{0})$, so lemma~\ref{lem:contract-delete} shows that $M/e\cong \UM{r-1}{r-1}\oplus\UM{\ell}{0}$. Contraction cannot create coloops, so the $r-1$ coloops of $M/e$ must already be present in $M$, that is, $M\cong\UM{r-1}{r-1}\oplus N$ for some matroid $N$ containing $e$. Since $M$ and thus $N$ have no loops, and $N/e\cong\UM{\ell}{0}$ has only loops, $N\cong\UM{\ell+1}{1}$ must be the bond. This proves $M\cong \UM{r-1}{r-1}\oplus\UM{\ell+1}{1}$.
\end{proof}
Note that $a_1(M)=\sum_{e\in M} (r-1)=n(r-1)$ for \emph{every} loopless matroid, so we cannot conclude anything about $M$ from saturating the bound of lemma~\ref{lem:upper-noloops} in the case $k=1$.

The Tutte polynomial of the dual matroid satisfies $\Tutte_{M^\star}(x,y)=\Tutte_M(y,x)$. In other words, $R_{ij}(M)=R_{ji}(M^\star)$. Inserting this relation into \eqref{eq:ak}, we see that
\begin{equation*}
    a_{n-k}(M^\star)=a_{k}(M)+(k-r)\binom{n}{k}
\end{equation*}
for all $k$, where $r=\rk(M)$ and $n=\abs{M}$. We can thus express lemma~\ref{lem:upper-noloops} in terms of the dual matroid, and we obtain the following equivalent statement.
\begin{lemma}\label{lem:upper-nocoloops}
    Let $M$ be a coloopless matroid of size $n$ and rank $r<n$. Then for all $k\geq0$,
    \begin{equation*}
        a_k(M)\leq a_k(\UM{r+1}{r}\oplus \UM{\ell-1}{0})
        = r\binom{n-1}{k} -\binom{n-1}{n-k}+ \binom{\ell-1}{n-k}
        \,.
    \end{equation*}
    Furthermore, equality at any $1\leq k\leq n-2$ implies that $M\cong\UM{r+1}{r}\oplus\UM{\ell-1}{0}$.
\end{lemma}
\begin{corollary}\label{cor:ehr0-coloopless}
    All matroids $M$ with size $n$, rank $r=n-\ell\geq 1$, and no loops satisfy $\ehr_M'(0)\geq \HN_{\ell}$. Equality holds if and only if $M\cong \UM{r-1}{r-1}\oplus \UM{1}{\ell+1}$.
    
    All matroids $M$ with size $n$, rank $r<n$, and no coloops satisfy $\ehr_M'(0)\geq \HN_r$. Equality holds if and only if $M\cong \UM{r+1}{r}\oplus\UM{\ell-1}{0}$.
\end{corollary}
\begin{proof}
    Combine \eqref{eq:ehr0-ak} with lemma~\ref{lem:upper-noloops} for the bound $\ehr_M'(0)\geq\ehr_{\UM{r-1}{r-1}\oplus\UM{\ell+1}{1}}'(0)$. The direct sum with $\UM{r-1}{r-1}$ only translates the matroid polytope, hence the Ehrhart coefficient is $\ehr'_{\UM{\ell+1}{1}}(0)=\HN_{\ell}$ as a special case of example~\ref{ex:ehr0-UM}. For $n\leq 2$, uniqueness of the minimizer is trivial since there is only one isomorphism class of loopless matroids in each rank at such $n$. For the uniqueness at sizes $n\geq 3$, we can invoke the equality clause of lemma~\ref{lem:upper-noloops} for $a_2(M)$, because $k=2<n$. The coloopless case goes analogously.
\end{proof}

We have thus determined the unique minimizers of $\ehr_M'(0)$ at each size and rank, within the classes of all matroids, loopless matroids, and coloopless matroids. We now turn to the class of connected matroids to prove Proposition~\ref{prop:MM}.
\begin{lemma}\label{lem:upper-conn}
    Let $M$ be a connected matroid of size $n\geq2$ and rank $r$. Then for all $k\geq0$,
    \begin{equation*}
        a_k(M)\leq a_k(\MM{n}{r})=r\binom{n-1}{k}-\binom{n}{k}+\binom{r}{k}+\binom{n-r}{n-k}-\delta_{k,r}
    \end{equation*}
    Here $\delta_{k,r}$ denotes Kronecker's delta: $\delta_{k,r}=1$ when $k=r$ and $\delta_{k,r}=0$ otherwise.
    
    Furthermore, equality $a_k(M)=a_k(\MM{n}{r})$ at any $2\leq k\leq n-2$ implies $M\cong\MM{n}{r}$.
\end{lemma}
\begin{proof}
    If $r=1$ or $r=n-1$, then $\MM{n}{r}\cong\UM{n}{r}$ and the claim reduces to lemma~\ref{lem:upper-noloops} or \ref{lem:upper-nocoloops}. Henceforth we assume that $2\leq r\leq n-2$ and thus $n\geq 4$.

    Pick any element $e\in M$. Then at least one of $M\setminus e$ and $M/e$ is connected, see \cite[\S6.5]{Tutte:ConnectivityMatroids}. Suppose that $M\setminus e$ is connected. Since $M$ is connected and $n=\abs{M}\geq 2$, $M$ and therefore also $M/e$ have no coloops. By induction on $n$ and lemmas~\ref{lem:contract-delete} and \ref{lem:upper-nocoloops}, we obtain
    \begin{align*}
        a_k(M) 
        &= a_k(M\setminus e)+a_{k-1}(M/e)
        \leq a_k(\MM{n-1}{r})+a_{k-1}(\UM{r}{r-1}\oplus \UM{\ell-1}{0}) \\
        &= r\binom{n-2}{k}-\binom{n-1}{k}+\binom{r}{k}+\binom{n-1-r}{n-1-k}-\delta_{k,r} \\
        &\quad+ (r-1)\binom{n-2}{k-1} -\binom{n-2}{n-k}+ \binom{n-r-1}{n-k} \\
        &= r\binom{n-1}{k}-\binom{n}{k}+\binom{r}{k}+\binom{n-r}{n-k}-\delta_{k,r}
        \,.
    \end{align*}
    For the minimal matroid $M=\MM{n}{r}$ and suitable $e$, we get $M\setminus e\cong \MM{n-1}{r}$ and $M/e\cong \UM{r}{r-1}\oplus\UM{\ell-1}{0}$, and thus the inequality becomes an equality. If $M\setminus e$ is not connected, then $M/e$ must be connected. Because $M\setminus e$ is loopless, we can then proceed via
    \begin{equation}\tag{$\ast$}\label{eq:conn-proof-deldisc}
        a_k(M) 
        = a_k(M\setminus e)+a_{k-1}(M/e)
        \leq a_k(\UM{r-1}{r-1}\oplus \UM{\ell}{1})+a_{k-1}(\MM{n-1}{r-1})
    \end{equation}
    and a very similar calculation as before reaches the same conclusion. This concludes the proof of the bound $a_k(M)\leq a_k(\MM{n}{r})$ and the formula for $a_k(\MM{n}{r})$.

    It remains to show that $a_k(M)=a_k(\MM{n}{r})$ for any $2\leq k\leq n-2$ implies $M\cong \MM{n}{r}$. Suppose that $M\setminus e$ is connected, hence $M/e$ free of coloops. To saturate the inequality from above, we must have $a_{k-1}(M/e)=a_{k-1}(\UM{r}{r-1}\oplus \UM{\ell-1}{0})$. By lemma~\ref{lem:upper-nocoloops}, this implies $M/e\cong \UM{r}{r-1}\oplus \UM{\ell-1}{0}$, which has $\ell-1=n-r-1\geq 1$ loops. Since $M$ has no loops, the loops in $M/e$ arise only from elements parallel to $e$. Therefore, $M$ must have a parallel class $P\subset M$ of size $\ell$ that contains $e\in P$. Its complement $B=M\setminus P$ has size $r$ and thus $\rk(M|_B)\leq r$. Connectivity of $M$ requires $r<\rk(M|_P)+\rk(M|_B)$, so $\rk(M|_B)> r-1$ and we conclude that $B$ is a basis of $M$. Since $M$ has no coloops, every $i\in B$ lies in some circuit $C_i\subseteq M$. Since $B$ is independent, we must have $C_i\cap P\neq\emptyset$, and since we can swap any element of $P$ for $e$ we may choose the circuit $C_i$ such that $C_i\subseteq B\sqcup\{e\}$. But $B\sqcup \{e\}$ has rank one and contains a \emph{unique} circuit, so $C_i=C$ is the same for all $i\in B$. It follows that $C=B\sqcup\{e\}\cong\UM{r+1}{r}$ is a circuit of $M$ and thus $M$ arises from $C$ by parallel extension of $e$ to $P$. In other words, $M\cong \UM{r+1}{r}\oplus_2 \UM{\ell+1}{1}\cong \MM{n}{r}$.

    If $M\setminus e$ is not connected, then to saturate the inequality \eqref{eq:conn-proof-deldisc} we need in particular that $a_k(M\setminus e)=a_k(\UM{r-1}{r-1}\oplus\UM{\ell}{1})$. Lemma~\ref{lem:upper-noloops} gives $M\setminus e\cong \UM{r-1}{r-1}\oplus\UM{\ell}{1}$ and we conclude $M\cong\MM{n}{r}\cong\MM{n}{n-r}^\star$ by duality from the discussion in the previous paragraph.
\end{proof}

\begin{proof}[{Proof of Proposition~\ref{prop:MM}}]
    For $r=1$ or $r=n-1$, the claim reduces to Corollary~\ref{cor:ehr0-coloopless}. For $2\leq r\leq n-2$, the bound $\ehr_M'(0)\geq\ehr_{\MM{n}{r}}'(0)$ follows from \eqref{eq:ehr0-ak} and lemma~\ref{lem:upper-conn}, and since e.g.\ $a_2(M)$ contributes with strictly negative coefficient, the equality $\ehr_M'(0)=\ehr_{\MM{n}{r}}'(0)$ requires that $a_2(M)=a_2(\MM{n}{r})$ and therefore $M\cong\MM{n}{r}$. All that remains is to confirm the actual value of the linear Ehrhart coefficient of the minimal matroids. By \eqref{eq:ehr0-ak},
    \begin{align*}
        \ehr'_{\MM{n}{r}}(0) 
        &= r\HN_{n-1} - \sum_{k=1}^{n-1} \frac{1}{(n-1)\binom{n-2}{k-1}} \left[ r\binom{n-1}{k}-\binom{n}{k}+\binom{r}{k}+\binom{n-r}{n-k}-\delta_{k,r} \right] \\
        &=2\HN_{n-1}+\frac{1}{(n-1)\binom{n-2}{r-1}} -\sum_{k=1}^{n-1} \frac{1}{(n-1)\binom{n-2}{k-1}} \left[ \binom{r}{k}+\binom{n-r}{k} \right] 
    \end{align*}
    where $r\HN_{n-1}$ cancelled the sum of $r\binom{n-1}{k}/\big[(n-1)\binom{n-2}{k-1}\big]=r/k$, $2\HN_{n-1}$ arose from the sum of $\binom{n}{k}/\big[(n-1)\binom{n-2}{k-1}\big]$ as in \eqref{eq:simple-harmonic-sum}, and we renamed $k\mapsto n-k$ in the sum with $\binom{n-r}{n-k}$. Now the claimed formula (which one could also have obtained from \cite[Theorem~1.6]{Ferroni:EhrhartMinimal})
    \begin{equation*}
        \ehr'_{\MM{n}{r}}(0)=\HN_{r-1}+\HN_{\ell-1}+\frac{1}{(n-1)\binom{n-2}{r-1}}
    \end{equation*}
    follows from the lemma below. This concludes the proof of Proposition~\ref{prop:MM}.
\end{proof}
\begin{lemma}
    For all integers $n>r>0$, we have the identity
    \begin{equation*}
        \sum_{k=1}^r \frac{1}{(n-1)\binom{n-2}{k-1}}\binom{r}{k} = \frac{1}{n-1}+\ldots+\frac{1}{n-r}=\HN_{n-1}-\HN_{n-1-r}.
    \end{equation*}
\end{lemma}
\begin{proof}
    Let $S_{n,r}$ denote the sum on the left-hand side. At $r=1$, the sum collapses to the term $k=1$ and $S_{n,1}=\frac{1}{n-1}$ is trivial. For $r\geq 2$ we can use $\binom{r}{k}=\binom{r-1}{k}+\binom{r-1}{k-1}$ and $\binom{n-2}{r-1}\binom{r-1}{k-1}=\binom{n-2}{k-1}\binom{n-1-k}{r-k}$ to obtain the recurrence
    \begin{equation*}
        S_{n,r}-S_{n,r-1}
        = \frac{1}{n-1}\sum_{k=1}^{r} \frac{\binom{r-1}{k-1}}{\binom{n-2}{k-1}}
        = \frac{1}{(n-1)\binom{n-2}{r-1}}\sum_{k=1}^{r} \binom{n-1-k}{n-r-1}
        =\frac{1}{n-r}
    \end{equation*}
    from the hockey stick identity. Inductively, this proofs the lemma.
\end{proof}

\bibliography{refs}

\end{document}